\documentclass{amsart}

\usepackage{color}
\usepackage{mathtools}
\usepackage{nicefrac}
\mathtoolsset{showonlyrefs} 
\usepackage{ mathrsfs }

\newtheorem{thm}{Theorem}[section]

\newtheorem{lema}[thm]{Lemma}

\theoremstyle{definition}
\newtheorem{defn}[thm]{Definition}
\theoremstyle{remark}
\newtheorem{rem}[thm]{Remark}

\numberwithin{equation}{section}

\newcommand{\R}{\mathbb R}
\newcommand{\N}{\mathbb N}
\newcommand{\lp}{L^{p}(\Omega)}
\newcommand{\wsp}{W^{s,p}(\Omega)}
\newcommand{\wup}{W^{1,p}(\Omega)}
\newcommand{\K}{\mathcal{K}}
\newcommand{\LL}{\mathcal{\mathscr{L}}}
\newcommand{\ve}{\varepsilon}
\newcommand{\lam}{\lambda}

\author[L. M. Del Pezzo and A. M. Salort]
{Leandro M. Del Pezzo and Ariel M. Salort}

\address{Leandro M. Del Pezzo and Ariel M. Salort
\hfill\break\indent
CONICET and Departamento  de Matem{\'a}tica, FCEyN,
Universidad de Buenos Aires,
\hfill\break\indent Pabellon I, Ciudad Universitaria (1428),
Buenos Aires, Argentina.}

\email{{\tt ldpezzo@dm.uba.ar,
asalort@dm.uba.ar
}}

\title{The first non-zero Neumann $p-$fractional eigenvalue}


\begin{document}
\begin{abstract}
 In this work we study the asymptotic behavior 
of the first non-zero Neumann 
$p-$fractional eigenvalue $\lambda_1(s,p)$ as $s\to 1^-$ and as $p\to\infty.$ 
We show that there exists a constant $\K$ such that
$\K(1-s)\lambda_1(s,p)$ goes to the first non-zero Neumann 
eigenvalue of the $p-$Laplacian. While in the limit case 
$p\to \infty,$ we prove that $\lambda_1(1,s)^{1/p}$ goes
to an eigenvalue of the H\"older $\infty-$Laplacian. 
\end{abstract}

\maketitle
\section{Introduction}

In this paper we set out to study the following non-local Neumann eigenvalue 
problems in a smooth bounded domain $\Omega\subset\R^n$ ($n\geq 1$)
\begin{equation} \label{ecu1}
	\begin{cases}
	-	\LL_{s,p}u = \lam |u|^{p-2}u \quad \mbox{ in } \Omega, \\
		u\in W^{s,p}(\Omega),
	\end{cases}
\end{equation}
where $1<p<\infty$ and $0<s<1.$ Here  
$\lam$ stands for the eigenvalue 
and $\LL_{s,p}$ is the regional fractional $p-$Laplacian, that is
\[
	\LL_{s,p}u(x)\coloneqq 2\mbox{ p.v.}\int_\Omega 
	\frac{|u(y)-u(x)|^{p-2}(u(y)-u(x))}{|x-y|^{n+sp}} 
	\, dy,
\]
  where p.v. is a commonly used abbreviation for ``in the principal value sense". 

Observe that, in the case $p=2,$ 
$\LL_{s,2}$ is the linear operator defined in  \cite{GMZ}, that is
the regional fractional Laplacian. 

The first non-zero eigenvalue of \eqref{ecu1} can be
characterized as 
\begin{equation*}
            \lambda_1(s,p)\coloneqq
            \inf\left\{
            \dfrac{\displaystyle
            \int_\Omega \int_\Omega \frac{|u(x)-u(y)|^p}{|x-y|^{n+sp}} \, 
	  dx \,dy}{ \displaystyle\int_\Omega|u(x)|^p\, dx}\colon u\in 
	  \mathcal{X}_{s,p}
            \right\},
\end{equation*}
where $\mathcal{X}_{s,p}=\left\{v\in\wsp\colon
            v\neq0,\int_\Omega |v(x)|^{p-2}v(x)\, dx=0\right\}.$
Here $W^{s,p}(\Omega)$ denotes 
a fractional Sobolev space (see Section \ref{pre}).
\bigskip

Non-local eigenvalue problems were recently studied in several 
papers. In \cite{rossi1} it was analyzed the first Neumann eigenvalue of a 
non-local diffusion problem for some non-singular convolution type operators. In 
\cite{rossi3} this analysis was extended for  non-local $p-$Laplacian type 
diffusion equations. Some properties about the first eigenvalue of the 
fractional Dirichlet $p-$Laplacian were established in \cite{FP,LL} and up to our knowledge no investigations were made about fractional Neumann eigenvalues.

\bigskip 

To be more concrete, we will study 
the asymptotic behavior of the first non-zero eigenvalue $\lambda_1(s,p)$ 
as $s\to 1^-$ and as $p\to \infty.$

\medskip

In order to introduce our results, we need to mention the well-known result
of Bourgain, Br\'ezis and Mironescu \cite{bourgain}: for any smooth bounded 
domain $\Omega\subset\R^n,$ $u\in W^{1,p}(\Omega)$ with 
$1< p<\infty$ there exists a constant $\K=\K(n,p,\Omega)$ such that
\begin{equation}\label{eq:auxint}
	\lim_{s\to 1^-} \K(1-s)\int_{\Omega}\int_{\Omega}
	\dfrac{|u(x)-u(y)|^p}{|x-y|^{n+sp}}\, dxdy=\int_{\Omega}|\nabla u|\, dx.
\end{equation}
See Theorem \ref{teo:bbm1} for more details.

\bigskip

Our first result is related to the limit as $s\to 1^-$ of  
 $\lam_1(s,p).$ We show that such that $\K(1-s)\lam_1(s,p)$  
goes to 
\begin{equation*}
  \lambda_1(1,p)\coloneqq
            \inf\left\{
            \dfrac{\|\nabla u\|^p_{\lp}}
            {\|u\|^p_{\lp}}\colon v\in 
	  \mathcal{X}_{1,p}
            \right\},
\end{equation*}
that is, the first non-zero eigenvalue  of the $p-$Laplacian 
with Neumann boundary conditions, namely $\lam_1(1,p)$ is the 
first non-zero eigenvalue of
\begin{equation} \label{ecu1-neuintro}
	\begin{cases}
		-\Delta_p u = \lam |u|^{p-2}u &\quad \mbox{ in } \Omega, \\
		\frac{\partial u}{\partial \nu} =0 &\quad \mbox{ on } 
		\partial \Omega.
	\end{cases}
\end{equation}
where $\Delta_p u= \mathrm{div} (|\nabla u|^{p-2}\nabla u)$ is the usual 
$p-$Laplacian and $\nu$ is the outer unit normal to $\partial \Omega$.

\begin{thm}\label{teo:conver1}
	Let $\Omega$ be a smooth bounded domain in $\R^n,$ and 
	$p\in(1,\infty).$ Then
	\[
		\lim_{s\to 1^-}  \K(1-s) \lam_1(s,p) =\lam_1(1,p),
	\]
	where   $\K$ is the constant in \eqref{eq:auxint}. 
\end{thm}

\bigskip

Lastly we study the limit case  $p\to\infty$. 
We show that
\begin{align*} 
	\lam_1(s,\infty)\coloneqq \lim_{p\to\infty} \lam_1(s,p)^{\frac{1}{p}
			}=\frac{2}{\mathrm{diam}(\Omega)^s}.	
\end{align*}	
Here $\mathrm{diam}(\Omega)$  
denotes diameter of $\Omega$, that is
\[
	\mathrm{diam}(\Omega)= \sup_{x,y\in\Omega} |x-y|.
\]

This result is truly different than that obtained in the local case, in contrast 
with the Dirichlet $p$-fractional Laplacian.  
More precisely, in \cite{RS} the authors show that
$$
	\lam_1(1,\infty)=
	\lim_{p\to\infty} \lam_1(1,p)^\frac{1}{p}  = 
	\frac{2}{\mathrm{diam}_{\Omega}(\Omega)},
$$
where 
$$
	\lambda_1(1,\infty)\coloneqq
	\inf\left\{\|\nabla u\|_{L^{\infty}(\Omega)}\colon
	u\in W^{1,\infty}(\Omega) \mbox{ s.t. } \max_{\Omega} u
	=-\min_{\Omega} u=1
	\right\},
$$
and $\mathrm{diam}_{\Omega}(\Omega)$ is the intrinsic diameter of $\Omega$, 
that is
\[
	\mathrm{diam}_{\Omega}(\Omega)= \sup_{x,y\in\Omega} d_{\Omega}(x,y)
\]
with $d_\Omega$ denoting the geodesic distance in $\Omega$.
Moreover, they show that if $u_p$ is a normalized minimizer of
$\lambda_1(1,p),$ then up to a subsequence, $u_p$ converge in 
$C(\overline{\Omega})$ to some minimizer 
$u\in W^{1,\infty}(\Omega)$ of $\lambda_1(1,\infty)$ which is a
solution of  
\[
	\begin{cases}
		\max \left \{ \Delta_{\infty} u, -|\nabla u|
		+\lambda_1(1,\infty)u \right \} &\text{ in }\{x\in\Omega\colon
		u(x)>0\},\\	
		\min \left \{ \Delta_{\infty} u, |\nabla u|
		+\lambda_1(1,\infty)u \right \} &\text{ in }\{x\in\Omega\colon
		u(x)<0\},\\	
		\Delta_{\infty} u=0  &\text{ in }\{x\in\Omega\colon
		u(x)=0\},\\
		\dfrac{\partial u}{\partial \nu}=0 &\text{ on }\partial \Omega,
	\end{cases}
\]
in the viscosity sense, where $\Delta_\infty$ is the $\infty-$Laplacian, 
that is 
\[
	\Delta_{\infty} u = -\sum_{i,j=1}^N \dfrac{\partial u}{\partial x_j}
	\dfrac{\partial^2 u}{\partial x_j\partial x_i}
	\dfrac{\partial u}{\partial x_j}.
\]
See also \cite{inftyN}.

\medskip

For the local Dirichlet $p-$Lapalcian eigenvalue problem the same limit
was studied in \cite{inftyD1, inftyD2}, where
the authors show that
\[
	\lim_{p\to\infty} \mu_1(1,p)^\frac{1}{p}= 
	\frac{1}{R(\Omega)}
	=\mu_1(1,\infty)\coloneqq\inf\left\{\dfrac{\|\nabla u\|_{L^{\infty}(\Omega)}}
	{\|u\|_{L^{\infty}(\Omega)}}\colon u\in W^{1,\infty}_0(\Omega),
	u\neq0\right\}.
\]
Here $R(\Omega)$ denotes the inradius (the radius of the largest ball 
contained in $\Omega$) and
$\mu_1(1,p)$ is the first  eigenvalue of the Dirichlet 
$p-$Laplacian. In addition,
they prove that the positive normalized eigenfunction $v_p$ associated to 
$\mu(1,p)$ converge, up to a subsequence, to a positive function 
$v\in W^{1,\infty}_0(\Omega)$ which is a minimizer of 
$\mu(1,\infty)$ and is a viscosity solution of
\[	
	\begin{cases}
		\min\{|Du|-\mu_1(1,\infty),\Delta_{\infty} u\}=0 
		&\text{ in }\Omega,\\
		u=0 &\text{ on }\partial\Omega.	
	\end{cases}
\]

Recently, the Dirichlet fractional $p-$Laplacian is considered, 
in \cite{LL} it was proved that
$$
	\lim_{p\to\infty} \mu_1(s,p)^\frac{1}{p} 
	= \frac{1}{R(\Omega)^s}
	=\mu_1(s,\infty)\coloneqq
	\inf\left\{\dfrac{[\phi]_{W^{s,\infty}(\Omega)}}
	{\|\phi\|_{L^{\infty}(\Omega)}}\colon \phi\in C_0^{\infty}(\Omega),
	\phi\neq0\right\},
$$
where $\mu_1(s,p)$ is the first eigenvalue of the non-local eigenvalue problems
\[
	\begin{cases}
		2\displaystyle\int_{\R^n} 
			\dfrac{|u(x)-u(y)|^{p-2}(u(x)-u(y))}{|x-y|^{n+sp}} 
			\, dy +\lam|u(x)|^{p-2}u(x)=0 &\textrm{ in } \Omega,\\[1em] 
		u\equiv 0 &\textrm { in } \R^n\setminus \Omega.
	\end{cases}
\]
Moreover, they show that if $w_{p}$ is a minimizer of $\mu_1(s,p),$  
then there exists $w\in C_0(\overline{\Omega})$ such that, up to
a subsequence $w_{p}\to w$ uniformly in  $\R^n$ which is a minimizer of 
$\mu_1(s,\infty)$ and is a solution of
\begin{gather*}
	\begin{cases}
	\max \left \{\mathcal{L}_\infty u(x), \mathcal{L}_\infty^- u(x)
	+\mu_1(s,\infty)u(x) \right \}=0 &\text{ in }\Omega,\\
	u=0 &\text{ on }\partial\Omega,	
	\end{cases}
\end{gather*}
in the viscosity sense. Here
\[
	\mathcal{L}_\infty u(x)\coloneqq
	\sup_{y\in\R^n}\dfrac{u(y)-u(x)}{|y-x|^s} 
	+\inf_{y\in\R^n}\dfrac{u(y)-u(x)}{|y-x|^s},
\]
and
\[
	\mathcal{L}_\infty^- u(x)\coloneqq
	\inf_{y\in\R^n}\dfrac{u(y)-u(x)}{|y-x|^s}.\\
\]

\medskip

In this context, our result is the following. 

\begin{thm}\label{teo:ptoinfty}
	Let $\Omega$ be  bounded open connected domain in $\R^n$ and 
	$s\in(0,1)$. Then
	\[
		\lim_{p\to\infty} \lam_1(s,p)^{\frac{1}{p}
		}=\frac{2}{\mathrm{diam}(\Omega)^s}=
		\lam_1(s,\infty)\coloneqq 
		\inf\left\{\dfrac{[u]_{W^{s,\infty}(\Omega)}}
		{\|u\|_{L^{\infty}(\Omega)}}\colon
		u\in \mathcal{A} \right\},	
	\]
	where $\mathcal{A}\coloneqq\left\{u\in W^{s,\infty}(\Omega)
	\colon u\neq0,
	\, \sup u+\inf u=0\right\}.$
	Moreover, if $u_p$ is the normalizer minimizer of
	$\lambda_1(s,p),$ then up to a subsequence 
	$u_p$ converges in $C(\overline{\Omega})$ to some minimizer 
	$u_\infty \in W^{s,\infty}(\Omega)$ of $\lambda_1(s,\infty)$
	which is a viscosity solution of  
	\begin{equation} \label{ecuvisc}
			\begin{cases}
				\max\{\LL_{s,\infty} u(x), \LL_{s,\infty}^- u(x) 
				+\lam_1(s,\infty) u(x)\}=0  & \mbox{ when }u(x)>0,\\
				\LL_{s,\infty} u(x)=0  & \mbox{ when }u(x)=0,\\
				\min\{\LL_{s,\infty} u(x), \LL_{s,\infty}^+ u(x) +
				\lam_1(s,\infty) u(x)\}
				=0  & \mbox{ when }u(x)<0,\\
			\end{cases}
	\end{equation}
	where $\LL_{s,\infty}u\coloneqq\LL_{s,\infty}^+u +\LL_{s,\infty}^-u,$ 
	$$
		\LL_{s,\infty}^+ u(x) \coloneqq
		\sup_{y\in\overline{\Omega}, y\neq x}
		\frac{u(y)-u(x)}{|y-x|^{s}}\quad \mbox{ and }\quad 
		\LL_{s,\infty}^- u(x) \coloneqq
		\inf_{y\in\overline{\Omega}, y\neq x}
		\frac{u(y)-u(x)}{|y-x|^{s}}.
	$$
\end{thm}

The operator $\LL_{s,\infty}$ is the H\"older $\infty-$Laplacian, see
\cite{CLM}.
\bigskip

Let us conclude the introduction with a brief comment on previous bibliography
that concerns mostly the non-local operators. 

\medskip

One of the biggest interests in defining the operator $\LL_{s,p}$ lies in its 
probabilistic interpretation in relation of a restricted type of L\'evy 
processes. 
In  \cite{BG}, it was studied the $s-$stable processes, a particular kind of 
L\'evy processes. For $s\in(0,1)$ and $n\geq 1$ they proved that the 
Dirichlet form associated with a symmetric $s-$stable  process in $\R^n$ is 
given by
$$
	\mathrm	{E}(u,v)=C\int_{\R^n} \int_{\R^n}
	\frac{(u(x)-u(y))(v(x)-v(y))}{|x-y|^{n+2s}} \, dx \, dy,
$$
where $u,v$ belong to $W^{s,2}(\R^n)$ and $C$ is a constant depending on $n$ 
and $s$. It is well known that $\mathrm	{E}$ is related to  
the fractional Laplacian $(-\Delta)^s$, that is 
\[
	(-\Delta)^s u(x)=C\mbox{ p.v.}\int_{\mathbb{R}^n} 
	\frac{u(x)-u(y)}{|x-y|^{n+2s}} 
	\, dy\quad \forall u\in W^{s,p}(\mathbb{R}^n)
\]
where $C$ is a constant depending on $n$ and $s,$  
precisely given by
\[
	C=\left(\int_{\R^n}
	\dfrac{1-\cos(\xi_1)}{|\xi|^{n+2s}} d\xi\right)^{-1},
\]
see \cite[Section 3]{DNPV}.

Due to the action of the process in the whole space it was widely used to model 
systems of stochastic dynamics with applications in operation research, queuing 
theory, mathematical finance among others, see \cite{AFT, Ap, CST} for instance.

\medskip

If one wished to restrict the action of a process to a bounded domain $\Omega
\subset \R^n$, one could consider the so-called \emph{$s-$stable process killed 
when leaving $\Omega$}, in which the Dirichlet form still being the same, but 
the functions are taken with support in $\Omega$, see \cite{neu2}.

\medskip

Alternatively, another way is to study the so-called 
\emph{censored stable process}, that is a stable process in which the jumps 
between $\Omega$ and its complement are forbidden. 
In this case, the functions are taken in the fractional Sobolev space 
$W^{s,2}(\Omega)$ and the correspondent Dirichlet form is given by
$$
	\mathcal{E}(u,v)=C\int_{\Omega} \int_{\Omega} \frac{(u(x)-u(y)(v(x)-v(y))}{|x-y|^{n+2s}} \, dx \, dy.
$$
This kind of processes are generated by 
\begin{equation} \label{regional}
	\Delta^s_\Omega u(x)=C
	\mbox{ p.v.}\int_{\Omega} \frac{u(x)-u(y)}{|x-y|^{n+2s}} \, dy.
\end{equation}
which   is called \emph{regional fractional Laplacian} in $\Omega$. 
See \cite{neu2,chen,aaa1,dida,GMZ} and references therein. 


\medskip

From a physical point of view, this operator describes a 
particle jumping from one point $x\in\Omega$ to another point $y\in\Omega$ with intensity proportional to $|x-y|^{-n-2s}$. Moreover, this kind of process can be used to describe some random flow in a closed domain with free action on the boundary, and they are always connected to the Neumann boundary problems. As it was pointed in \cite{rossi1,rossi2} the idea of $s-$process in which its jumps from $\Omega$ to the complement of $\Omega$ are suppressed, are related to the Neumann non-local evolution equation 
\begin{align}
\begin{cases}
	u_t(x,t)=\Delta^s_\Omega u(x)\\
	u\in W^{s,2}(\Omega)
\end{cases}
\end{align}
since the individuals are ``forced" to stay inside $\Omega$. In contrast with 
the classical heat equation $u_t=\Delta u$, the diffusion of the density $u$ at 
a point $x$ and a time $t$ depends not only on $u(x,t)$, but also on all values 
of $u$ in a neighborhood of $x$.

\bigskip

In the course of the writing of this paper, the authors in \cite{DRV} introduced 
a new Neumann problem for the fractional Laplacian by considering the non-local 
prescription 
\[
	\mbox{p.v.} \int_\Omega \dfrac{u(x)-u(y)}{|x-y|^{n+2s}}\, dy=0
\] for $x\in \R^n\setminus \Omega$ as a generalization of the classical 
Neumann condition $\partial_\nu u=0$ on $\partial \Omega$.

\bigskip

The paper is organized as follows: in Section \ref{pre} we collect some 
preliminaries; in Section \ref{autovalor} we deal with the first 
non-zero eigenvalue; in Section \ref{sto1} we prove Theorem \ref{teo:conver1};
in Section \ref{pinfty} we prove Theorem \ref{teo:ptoinfty},
while in the final section we give an example of non-linear non-local 
operator such that its first non-zero eigenvalue $\mu(s,p)$ has 
the following property: $\mu(s,p)^{\nicefrac1p}\to\nicefrac{2}
{\mathrm{diam}_\Omega(\Omega)}$ as $p\to\infty.$

\section{Preliminaries}\label{pre}

We begin by recalling some results concerning the fractional Sobolev 
spaces.

\medskip

Let $\Omega$ be an open set in $\R^n,$ $s\in(0,1)$ and $p\in[1,\infty).$
The fractional Sobolev spaces is defined as
\[
	W^{s,p}(\Omega)\coloneqq
	\left\{u\in L^p(\Omega) \colon 
	\frac{|u(x)-u(y)|}{|x-y|^{n/p+s}}\in L^p(\Omega\times\Omega)\right\},
\] 
which endowed with the norm
\[
	\|u\|^p_{W^{s,p}(\Omega)}\coloneqq
	  \|u\|_{L^p(\Omega)}^p + 
	  \int_\Omega \int_\Omega \frac{|u(x)-u(y)|^p}{|x-y|^{n+sp}} \, 
	  dx \,dy,
\]
is a separable Banach space. Moreover, if $p\in(1,\infty)$ then
$\wsp$ is reflexive.

The fractional space $W^{s,\infty}(\Omega)$ is defined as the space of functions
\[
	W^{s,\infty}(\Omega)\coloneqq \left\{ u\in L^\infty(\Omega) \colon 
	\frac{u(x)-u(y)}{|x-y|^s} \in L^\infty(\Omega\times\Omega)\right\}
\]
with the norm
\[
	\|u\|_{W^{s,\infty}(\Omega)} \coloneqq \|u\|_{L^\infty(\Omega)} + \left
	\|\frac{u(x)-u(y)}{|x-y|^s} \right\|_{L^\infty(\Omega\times\Omega)}.
\]

Throughout the paper $[u]_{\wsp}$ denotes the  so-called Gagliardo seminorm 
\[
	[u]_{W^{s,p}(\Omega)}\coloneqq 
	\begin{cases}
		\left(\displaystyle\int_\Omega \int_\Omega 
		\dfrac{|u(x)-u(y)|^p}{|x-y|^{n+sp}} \, dx \,dy\right)^{\frac1p}, 
		&\text{ if }1\le p<\infty,\\[.35cm]
		\displaystyle\left
		\|\frac{u(x)-u(y)}{|x-y|^s} \right\|_{L^\infty(\Omega\times\Omega)}
		&\text{ if }p=\infty.
	\end{cases}
\]

  For more details related these spaces and their properties, 
see, for instance, \cite{Adams,DD,DNPV}. 

\medskip
 
The proof of the following lemma can be found in \cite{DD}.

\begin{lema}\label{lem:densidad}
	Let $\Omega\subset\R^n$ be an open set of class $C^1.$ Then
	$C^1(\overline{\Omega})$ is dense in $\wsp.$ 
\end{lema}
The next results are established in  
\cite[Corollaries 2 and 7]{bourgain}.

\begin{thm}
	\label{teo:bbm1}	
	Let $\Omega$ be a smooth bounded domain in $\R^n,$ and 
	$p\in(1,\infty).$
	Assume $u\in\lp,$ then
	\[
		\lim_{s\to 1^-}\K(1-s)[u]_{\wsp}^p=[u]_{\wup}^p 
	\]
	with
	\[
		[u]_{\wup}^p =\begin{cases}
			\displaystyle\int_{\Omega}|\nabla u|^p \, dx, 
				&\text{ if } u\in\wup,\\
				\infty &\text{ if } u\notin\wup.
			\end{cases}	
	\]
	Here $\K$ depends only the $p$ and $\Omega.$ 	
\end{thm}

 This result was later completed in \cite{MR1940355}, where the authors show that
for $u\in\bigcup_{s\in(0,1)}$ $W_0^{s,p}(\R^n)$ with $1\le p<\infty,$ we 
have that 
\[
	\lim_{s\to 0^+} \frac{sp}{2\omega_{n-1}}[u]_{W^{s,p}(\R^n)}^p
	=\|u\|_{L^p(\R^n)}^p.
\] 
Here the space $W_0^{s,p}(\R^n)$ is the closure of $C_0^\infty(\R^n)$
in the norm $[u]_{W^{s,p}(\R^n)}$ and $\omega_{n-1}$ is the $(n-1)$-dimensional
Hausdorff measure of the unit sphere $S^{n-1}.$

Finally in \cite{MR2146631}, the author shows that the above two result  
can be viewed as consequences of continuity principles for real interpolation scales.
 
\begin{thm}
	\label{teo:bbm2}	
	Let $\Omega$ be a smooth bounded domain in $\R^n,$  and
	$p\in(1,\infty).$ Let $\{u_s\}_{s\in(0,1)}$ be a subset of $L^p(\Omega)$
	 such that for any $s\in(0,1)$ we have that
	$u_s\in\wsp$ and
	\[
		(1-s)[u_s]_{\wsp}\le C.
	\]
	Then, there exist $u\in\wup$ and a subsequence 
	$\{u_{s_k}\}_{k\in\N}$ such that
	\begin{align*}
		u_{s_k}\to u &\quad\mbox{strongly in } \lp,\\
		u_{s_k}\rightharpoonup u&\quad\mbox{weakly in }
		 W^{1-\varepsilon,p}(\Omega),
	\end{align*}
	for all $\varepsilon> 0.$ 
\end{thm}

\begin{rem}
	In \cite{bourgain} some inequalities involving fractional integrals 
	are established. 
	A carefully computation allows us to compute explicitly the constant in 
	\cite[Lemma 2]{bourgain}. By means of the Chebyshev inequality together with 
	Lemma 2 from \cite{bourgain}, in equation (36) from \cite{bourgain} 
	it is obtained that
	\[
		\ve[u_\ve]_{W^{1-\ve,p}(\Omega)}^p \geq 2^{-p\delta  }\delta
		[u_\ve]_{W^{1-\delta,p}(\Omega)}^p,
	\]
	where $0<\ve<\delta$. 

	Denoting $s:=1-\ve$ and $t:=1-\delta$, last inequality is equivalent to
	\begin{equation} \label{desig}
		(1-t)[u_s]_{W^{t,p}(\Omega)}^p 
		\leq 2^{p(1-t)}(1-s)[u_s]_{W^{s,p}(\Omega)}^p,
	\end{equation}
	where $0<t<s<1$.
\end{rem}

\medskip

For any $s\in(0, 1)$ and any $p \in [1, \infty),$ 
we say that an open set $\Omega\subset\mathbb{R}^n$ admits an
$(s,p)$-extension domain if there exists a positive constant 
$C = C(n, p, s, \Omega)$ such that: for every function 
$u \in W^{s,p}(\Omega)$ there exists $\tilde{u}\in W^{s,p}(\R^n)$ with 
$\tilde{u}(x)=u(x)$  for all $x\in \Omega$ 
and $\|\tilde{u}\|_{W^{s,p}(\mathbb{R}^n)}\le C\|u\|_{W^{s,p}(\Omega)}.$
For example, any Lipschitz open set $\Omega$ admits a $(s,p)$-extension,
see \cite[Proposition 4.43]{DD}.

\medskip
 
A  useful result to be used is the fractional compact embeddings. 
For the proof see \cite[Corolary 7.2]{DNPV}
and \cite[Theorem 4.54]{DD}.

\begin{thm}\label{teo:inclucomp} 
	Let $s\in(0,1),$ $p\in(1,\infty)$ and 
	$\Omega\subset\mathbb{R}^n$ be a bounded open set that admits 
	an $(s,p)$-extension.  If $sp<n$ then we have the following compact 
	embeddings
	\[
		\wsp\hookrightarrow L^q(\Omega) 
	\qquad \mbox{ for all } q\in[1,p_s^\star).
	\]
	
	In addition, if $\Omega$ has a Lipschitz boundary and $sp\ge n$ then
	we have the following compact embeddings:
	\begin{align*}
	&\wsp\hookrightarrow L^q(\Omega) 
	\qquad \mbox{ for all } q\in[1,p_s^\star), &\mbox{ if } sp= n;\\
	&\wsp\hookrightarrow C^{0,\lambda}_b(\Omega) 
	\quad \mbox{ for all } 
	\lambda<s-\nicefrac{n}{p}, &\mbox{ if } sp>n.
	\end{align*}
	
	Here $p_s^\star$ is the fractional critical Sobolev exponent, that is
	\[
		p_s^\star\coloneqq
			\begin{cases}
				 \dfrac{np}{n-sp}, &\text{ if } sp<n,\\
				 \infty,  &\text{ if } sp\ge n.\\
			\end{cases}
	\]
\end{thm}

\section{The first non-zero eigenvalue} \label{autovalor}
Now we will show that $\lambda_1(s,p)$ is the first non-zero eigenvalue
of  \eqref{ecu1}.

\medskip

We say that  the value $\lambda\in\R$ is an eigenvalue of problem \eqref{ecu1} 
if there exists $u\in\wsp\setminus\{0\}$ such that
\begin{equation}\label{eq:euler1}
	\mathcal{E}(u,\phi) = \lam \int_\Omega 
	|u|^{p-2}(x)u(x)\phi(x)\, dx
	\quad \forall \phi\in C^1(\overline{\Omega}),
\end{equation}
where
\begin{equation} \label{eq:euler2}
	\mathcal{E}(u,\phi)\coloneqq\int_\Omega \int_\Omega 
	\frac{|u(y)-u(x)|^{p-2}(u(y)-u(x))(\phi(y)-\phi(x))}{|x-y|^{n+sp}} 
	\, dx \, dy.
\end{equation}
In which case, we say that $u$ is  an eigenfunction 
associated to $\lambda.$

\medskip

Of course $\lambda=0$ is an eigenvalue and it is isolated and simple. 
Moreover, if $\lambda>0$ is an eigenvalue
and $u$ is an eigenfunction associated to $\lambda,$
then, taking $\phi\equiv1$ as a test function in 
\eqref{eq:euler1}, we have
\[
	\int_{\Omega}|u(x)|^{p-2}u(x)\, dx = 0.
\]
Thus, the existence of the first non-zero eigenvalue 
$\lambda_1(s,p)$ of \eqref{ecu1} is related to the 
problem of minimizing the following non-local quotient
\[
	\frac{[v]^p_{W^{s,p}(\Omega)}}{\|v\|^p_{L^p(\Omega)}}
\]
among all functions $v\in W^{s,p}(\Omega)\setminus\{0\}$ such that
$\int_{\Omega}|v(x)|^{p-2}v(x)\,dx=0. 	$

\medskip

We begin establishing the following result.

\begin{thm}\label{teo:propaut}
      Let $\Omega$ be an open set of class $C^1,$ $s\in(0,1)$ 
      and $p\in(1,\infty).$ 
      Then
      \begin{equation}\label{autovsp}
            \lambda_1(s,p)=
            \inf\left\{\dfrac{[v]^p_{\wsp}}{\|v\|^p_{\lp}}\colon 
            v\in\wsp, 
            v\neq0,\int_\Omega |v(x)|^{p-2}v(x)\, dx=0\right\}
      \end{equation}
      is the first non-zero eigenvalue of \eqref{ecu1}. 
\end{thm}

\begin{proof}
     Let $\{u_j\}_{j\in\N}\subset\wsp$ be a minimizing sequence for 
     $\lambda_1(s,p)$ 
     such that $\|u_j\|_{\lp}=1$ for all $j\in\N.$ Then there exists a 
     constant $C$ such that
     \[
           [u_j]_{\wsp}\le C.
     \]
     Therefore $\{u_j\}_{j\in\N}$ is bounded in $\wsp.$ Then, by Theorem \ref{teo:inclucomp}, there exists
     a function $u\in\wsp$ such that, up to a subsequence that we still call 
     $\{u_j\}_{j\in\N},$
     \begin{align*}
           u_j\rightharpoonup u &\qquad\mbox{weakly in }\wsp,\\
           u_j\to u &\qquad\mbox{strongly in }\lp.
     \end{align*} 
     Hence $\|u\|_{\lp}=1,$ $|u_j(x)|^{p-2}u_j(x)\to|u(x)|^{p-2}u(x)$
     a.e. in $\Omega,$ and 
     \[
     	\||u_j|^{p-2}u_j\|_{L^{\nicefrac{p}{(p-1)}}(\Omega)}
     	\to \||u|^{p-2}u\|_{L^{\nicefrac{p}{(p-1)}}(\Omega)}.
     \]
     Then, by \cite[Theorem 12]{RR}, $|u_j|^{p-2}u_j\to |u|^{p-2}u$
     strongly in $L^{\nicefrac{p}{(p-1)}}(\Omega).$ Therefore,
     since $\int_\Omega |u_j(x)|^{p-2}u_j(x)\, dx=0$ for all $j\in\N,$ 
     we have that $\int_\Omega |u(x)|^{p-2}u(x)\, dx=0.$ Then $u$ is not
     constant.
     
     On the other hand, since 
     $u_j\rightharpoonup u $ weakly in $\wsp,$
     \[
           [u]_{\wsp}^p\le\liminf_{j\to\infty}[u_j]_{\wsp}^p
           =\lim_{j	\to\infty}[u_j]_{\wsp}^p
           =\lambda_{1}(s,p).
     \]
     Then, by \eqref{autovsp}, we have that
     \[
           [u]_{\wsp}^p=\lambda_1(s,p).
     \]
     Observe that $\lambda_1(s,p)>0$ due to $u$ is not
     constant. In addition, $\lambda_1(s,p)$ is attained in
     \[
     	\left\{v\in\wsp\colon \int_\Omega |v(x)|^{p-2}v(x)\, dx=0
     	\mbox{ and } \|v\|_{\lp}=1  \right\}.
     \]
     Then, proceeding as in the proof of Theorem 4.3.77 in \cite{PK},
     we have that $\lambda_1(s,p)$
     is the first non-zero 
     eigenvalue of \eqref{ecu1}. 
\end{proof}

\medskip

Finally we show that if an eigenfunction belongs to $C(\overline{\Omega})$
then it is a viscosity solution of
\begin{equation}\label{eq:visco}
	-\LL_{s,p} u =\lam_1(s,p) |u|^{p-2}u
\end{equation}
in the following sense.

\begin{defn} 
	Suppose that $u\in C(\overline{\Omega})$. We say that $u$ is a 
	\emph{viscosity super-solution} 
	(resp. \emph{viscosity sub-solution}) in 
	$\Omega$ of the equation \eqref{eq:visco} if the following holds: 
	whenever $x_0\in\Omega$ and $\varphi\in C^1(\overline{\Omega})$ 
	are such that
	$$
		\varphi(x_0)=u(x_0) \quad \mbox{ and } \quad \varphi(x)\leq u(x)  
		\mbox{ (resp. $\varphi(x)\geq u(x))$  for all } x\in \R^n
	$$
	then we have
	$$
		\LL_{s,p} \varphi(x_0) + \lam_1(s,p) 
		|\varphi(x_0)|^{p-2}\varphi(x_0)	 
		\leq 0 \quad \mbox{(resp. $\geq 0$)}.
	$$
	A \emph{viscosity solution} is defined as being 
	both a viscosity super-solution and a viscosity sub-solution.
\end{defn}

For the proof of the following theorem, see \cite[Proposition 11]{LL}.
\begin{thm}\label{debilviscosa} 
	Let $s\in(0,1)$ and $p\in(1,\infty)$ such that $s<1-\nicefrac1{p}$. An eigenfunction $u\in C(\overline{\Omega})$
	associated to $\lam_1(s,p)$ is a viscosity solution of \eqref{eq:visco}.
\end{thm}

\section{The limit as $s\to1^-$}\label{sto1}
In this section, our main aim is to prove that 
\[
	 \K(1-s)\lambda_1(s,p)\to\lambda_1(1,p)	\qquad \mbox{as } s\to1^-,
\]
where $\K$ is the constant of Theorem \ref{teo:bbm1}. 

\medskip

Before we prove Theorem \ref{teo:conver1}, we need to show the
following technical lemma.
\begin{lema} \label{lema1}
	Let $\{s_j\}_{j\in\N}\subset(0,1)$ and $\{u_j\}_{j\in\N}\subset\lp$  
	such that $s_j\to 1^-$ as $j\to\infty,$ 
	$u_j\in W^{s_j,p}(\Omega),$ 
	\begin{equation}
	\label{eq:lema1h}
		\K(1-s_j)[u_j]_{W^{s_j,p}(\Omega)}^p=1 \mbox{ and }
		\int_\Omega |u_j(x)|^{p-2}u_j(x)\, dx=0
	\end{equation}
	for all $j\in N$. Then there exist  subsequences
	$\{s_{j_k}\}_{k\in N}$ and $\{u_{j_k}\}_{k\in N}$, and a 
	function $u\in\wup$ such that 
	\[
		u_{j_k} \to u \quad \mbox{ strongly in } L^p(\Omega)
	\]
	and
	\[
		 [ u]^p_{\wup} \le
		 \liminf_{k\to\infty} \K(1-s_{j_k}) [u_{s_{j_k}}]^p_{W^{s_{j_k},p}(\Omega)} 
	\]
	with $\displaystyle\int_\Omega |u(x)|^{p-2}u(x)\, dx=0.$
\end{lema}
 
\begin{proof}
	For any $t\in(0,1),$
	there exists $j_0\in\N$ such that
	$0<t<s_j<1$ for all $j\ge j_0.$ 
	By \eqref{desig} and \eqref{eq:lema1h} 
	it follows that 
	\begin{equation}
	\label{eq:aux1l1}
			\K(1-t)[u_j]_{W^{t,p}(\Omega)}^p 
			\leq 2^{p(1-t)}\K(1-s_j)[u_j]_{W^{s_{j},p}(\Omega)}^p 
			\le 2^{p(1-t)}
			\quad\forall j\ge j_0.
	\end{equation}
	Then, by Theorem \ref{teo:inclucomp}, there exist 
	a subsequence $\{u_{j_k}\}_{k\in\N},$ 
	and a function $u\in\wup$ such that 
	\begin{align*}
		u_{j_k}\to u &\quad\mbox{strongly in } \lp,\\
		u_{j_k}\rightharpoonup u&\quad\mbox{weakly in }
		 W^{t,p}(\Omega).
	\end{align*}
	Using \eqref{eq:aux1l1}, we have
	\begin{align*}
		\K(1-t)[u]_{W^{t,p}(\Omega)}^p &\leq 
	 	\liminf_{k\to \infty} \K(1-s_{j_k})[u_{j_k}]_{W^{t,p}(\Omega)}^p\\ 
	 	&\leq  2^{p(1-t)}\liminf_{k\to \infty} \K(1-s_{j_k}) 
	 	[u_{j_k}]^p_{W^{s_{j_k},p}(\Omega)}.
	\end{align*}
	
	On the other hand, by Theorem \ref{teo:bbm1}, we get
	\[
		[u]_{\wup}^p=
		\lim_{t\to 1^-} \K(1-t)[u]_{W^{t,p}(\Omega)}^p 
		\leq  \liminf_{k\to\infty} \K(1-s_{j_k}) 
		[u_{j_k}]^p_{W^{s_{j_k},p}(\Omega)}.
	\]
	
	Finally, we show that $\int_\Omega |u(x)|^{p-2}u(x)\, dx=0.$ 
	We have that $|u_{j_k}(x)|^{p-2}u_{j_k}(x)\to|u(x)|^{p-2}u(x)$
     a.e. in $\Omega,$ and 
     \[
      	\||u_{j_k}|^{p-2}u_{j_k}\|_{L^{\nicefrac{p}{(p-1)}}(\Omega)}
     	\to \||u|^{p-2}u\|_{L^{\nicefrac{p}{(p-1)}}(\Omega)},
     \]
     due to $u_{j_k}\to u$ strongly in $\lp.$
     Then, by \cite[Theorem 12]{RR}, $|u_{j_k}|^{p-2}u_{j_k}\to |u|^{p-2}u$
     strongly in $L^{\nicefrac{p}{(p-1)}}(\Omega).$ Therefore,
     since $\int_\Omega |u_{j_k}(x)|^{p-2}u_{j_k}(x)\, dx=0$ for all $k,$ 
     we have that $\int_\Omega |u(x)|^{p-2}u(x)\, dx=0.$ 
\end{proof}

We finish this section by proving Theorem \ref{teo:conver1}.

\begin{proof}[Proof of Theorem \ref{teo:conver1}]
	Let $u\in W^{1,p}(\Omega)$ be an eigenfunction associated to $\lam_1(1,p).$
	Since $W^{1,p}(\Omega) \subset W^{s,p}(\Omega)$ for all $s\in (0,1)$
	and $\int_\Omega |u(x)|^{p-2}u(x)\, dx=0$,  
	$u$ is an admissible function in the variational characterization of 
	$\lam_1(s,p)$ for all $s\in(0,1)$. Then, 
	\[
		\K(1-s)\lam_{1}(s,p)\leq \K(1-s)\frac{[u]^p_{\wsp}}{\|u\|^p_{\lp}}.
	\]
	Therefore, by Theorem \ref{teo:bbm1}, we get that
	\begin{align} \label{cota1}
		\limsup_{s\to 1^-}\K(1-s)\lam_{1}(s,p)\leq 
		\lim_{s\to 1^-} \K(1-s)
		\frac{[u]^p_{W^{s,p}(\Omega)}}{\|u\|^p_{L^p(\Omega)}}=  
		\frac{[u]^p_{\wup}}{\|u\|^p_{L^p(\Omega)}}
		=\lam_1(1,p).
	\end{align}
	
	On the other hand, let $\{s_j\}_{j\in \N}$ 
	be a sequence in $(0,1)$ such that $s_j\to 1^-$ as 
	$j\to\infty$ and 
	\begin{equation}
	\label{eq:liminf}
		\lim_{j\to\infty}\K(1-s_j)\lambda_1(s_j,p)=
		\liminf_{s\to 1^-}\K(1-s_j)\lambda_1(s,p).
	\end{equation}
	For $j\in\N,$ let us choose $u_j\in\wsp$ 
	such that 
	\[
		\K(1-s_j)[u_j]_{W^{s_j,p}(\Omega)}^p=1,\quad
		\int_{\Omega}|u_j(x)|^{p-2}u_j(x)\, dx=0,
	\]
	and
	\[
		\K(1-s_{j})[u_j]_{W^{s_j,p}(\Omega)}^p=
		\K(1-s_{j})\lambda_1(s_j,p)\|u_{s_j}\|_{\lp}^p.
	\]
	By Lemma \ref{lema1}, 
	there exist a subsequence, still denote $\{u_j\}_{j\in\N},$ and 
	a function $u\in \wup$ such that 
	\[
		u_j\to u \mbox{ strongly in } \lp,
		\quad \int_{\Omega}|u(x)|^{p-2}u(x)\, dx=0,
	\]
	and
	\[
		[u]^p_{\wup}\le
		\liminf_{j\to\infty} \K(1-s_j) [u_{j}]^p_{W^{s_j,p}(\Omega)}.
	\]
	Therefore, $[ u]^p_{\wup}\le1.$
	Moreover, since 
	$$
		1=\K(1-s_{j})[u_j]_{W^{s_j,p}(\Omega)}^p
		=\K(1-s_{j})\lambda_1(s_j,p)\|u_j\|_{\lp}^p
	$$  
	for all $j\in\N$ and $u_j\to u$  strongly in $\lp,$
	 by \eqref{eq:liminf}, we have
	\begin{equation}
		\label{eq:liminf2}
		1=\liminf_{s\to1^-}\K(1-s)\lambda_1(s,p)\|u\|_{\lp}^p.
	\end{equation}
	Thus, $u$ is an admissible function in the variational 
	characterization of $\lam_1(1,p)$. Then, using that 
	$[ u]^p_{\wup}\le1$ and \eqref{eq:liminf2}, 
	we have that
	\begin{align} \label{cota2}
		\lam_1(1,p)\leq \liminf_{s\to1^-}\K(1-s)\lambda_1(s,p).
	\end{align}
	
	From \eqref{cota1} and \eqref{cota2} the result follows.
\end{proof}

\section{The limit as $p\to\infty$}\label{pinfty}

The goal of this section is to study the limit as $p\to \infty$ 
of the first non-zero eigenvalue $\lam_1(s,p).$ Before  beginning, 
we need to establish the following lemma.

\begin{lema}
	\label{lem:aux1} Let $\Omega$ be a bounded open and connected domain in $\R^n,$
	$s\in(0,1),$ $x_0\in\Omega$ and $c\in\mathbb{R}.$ 
	The function $w(x)=|x-x_0|-c$
	belongs to $W^{1,\infty}(\Omega)$ and
	\[
		[w]_{\wsp}\le 
		\dfrac{\kappa_n^{\frac1p}\mathrm{diam}(\Omega)^{1-s}
		|\Omega|^{\frac1p}}{(p(1-s))^{\frac1p}} \quad
		\forall p\in(1,\infty)
	\]
	where $\kappa_n$ is the measure of unit ball
	and $|\Omega|$ is the measure of $\Omega.$
\end{lema}
\begin{proof}
	We start the proof recalling that
	\[
		w\in W^{s,\infty}(\Omega)\quad\mbox{ and }
		\quad |w|_{W^{s,\infty}(\Omega)}= \mathrm{diam}(\Omega)^{1-s}
		\mbox{ a.e. in }\Omega.
	\]
	Then, we have that $w\in\wsp$ for all $p\in(1,\infty).$
	
	On the other hand
	\begin{align*}
		[w]_{\wsp}^p&= \int_\Omega\int_\Omega
		\dfrac{|w(x)-w(y)|^p}{|x-y|^{n+ps}}\,dx dy\\
		&=\int_\Omega\int_\Omega
		\dfrac{||x-x_0|-|y-x_0||^p}{|x-y|^{n+ps}}\,dx dy\\
		&\le
		\int_\Omega\int_{\Omega}|x-y|^{p(1-s)-n}\, dxdy\\
		&\le\dfrac{\kappa_n\mathrm{diam}(\Omega)^{p(1-s)}
		|\Omega|}{p(1-s)}.
	\end{align*}
	This proves the lemma.
\end{proof}

We carry out the proof of Theorem \ref{teo:ptoinfty} in the two following lemmas.

\begin{lema}\label{lema:ptoinfty1}
	Let $\Omega$ be a  bounded open and connected domain in $\R^n$ and 
	$s\in(0,1)$. Then
	\[
		\lim_{p\to\infty} \lam_1(s,p)^{\frac{1}{p}
		}=\frac{2}{\mathrm{diam}(\Omega)^s}=
		\lam_1(s,\infty)\coloneqq 
		\inf\left\{\dfrac{[u]_{W^{s,\infty}(\Omega)}}
		{\|u\|_{L^{\infty}(\Omega)}}\colon
		u\in \mathcal{A} \right\},	
	\]
	where $\mathcal{A}\coloneqq\left\{u\in W^{s,\infty}
	(\Omega)\colon u\neq0,
	\, \sup u+\inf u=0\right\}.$
	Moreover, if $u_p$ is the normalizer minimizer of
	$\lambda_1(1,p),$ then up to a subsequence, 
	$u_p$ converges in $C(\overline{\Omega})$ to some minimizer 
	$u_\infty \in W^{s,\infty}(\Omega)$ of $\lambda_1(1,\infty).$
\end{lema}

\begin{proof} 
	We split the proof in three steps.
	
	\bigskip
	
	{\it Step 1.} Let us prove that 
	\begin{equation}
		\label{eq:step1}
		\limsup_{p\to\infty} 
		\lam_1(s,p)^\frac{1}{p} \le \frac{2}{\mathrm{diam}(\Omega)^s}.
	\end{equation}
	
	Let $x_0\in \Omega$. We choose $c_p\in\R$ such that the 
	function
	$$
		w_p(x)=|x-x_0| - c_p
	$$
	satisfies that
 	$$
	\int_\Omega |w_p(x)|^{p-2}w_p(x) \, dx =0.
	$$
	We can also observe that $w_p\in W^{s,p}(\Omega)$ for all
	$p\in(1,\infty).$ Then, by Lemma \ref{lem:aux1}, 
	for any $p\in(1,\infty)$ we have that
	\[
		\lambda_1(s,p)
		\le \dfrac{\displaystyle\int_{\Omega}\int_{\Omega}
		\dfrac{|w_p(x)-w_p(y)|^p}{|x-y|^{n+sp}}
		\, dx dy}{\displaystyle\int_\Omega |w_p(x)|^{p} \, dx }
		\le \dfrac{\kappa_n^{\frac1p}\mathrm{diam}(\Omega)^{1-s}
		|\Omega|^{\frac1p}}{(p(1-s))^{\frac1p}
		\int_\Omega |w_p(x)|^{p} \, dx }.
	\]
	 Then
	 \begin{equation}
		\label{eq:pinfty1}
		\limsup_{p\to\infty}\lambda_1(s,p)^{\frac1p}
		\le\dfrac{\mathrm{diam}(\Omega)^{1-s}}
		{\displaystyle\liminf_{p\to\infty}
		\left(\int_\Omega |w_p(x)|^{p} \, dx\right)^{\frac1p} }.
	\end{equation}
	
	On the other hand, proceeding as in the proof of 
	Lemma 1 in \cite{inftyN}, we have that
	\begin{equation}
	\label{eq:pinfty2}
		\liminf_{p\to\infty}
		\left(\int_\Omega |w_p(x)|^{p} \, dx\right)^{\frac1p} 
		\ge \dfrac{\mathrm{dia  m}(\Omega)}2.
	\end{equation}
	
	Thus, by \eqref{eq:pinfty1} and \eqref{eq:pinfty2}, we have that
	\eqref{eq:step1} holds.
	
	\bigskip
	
	{\it Step 2.} Let us prove that 
	\[
		\inf\left\{\dfrac{[u]_{W^{s,\infty}(\Omega)}}
		{\|u\|_{L^{\infty}(\Omega)}}\colon
		u\in \mathcal{A} \right\}\le \liminf_{p\to \infty}
		\lambda_1(s,p)^{\frac{1}{p}}.
	\] 

	Let $\{p_j\}_{j\in\mathbb{N}}$ be an increasing sequence in $(1,\infty)$ 
	and $\{u_{j}\}_{j\in\mathbb{N}}$ be a sequence of measurable functions 
	such that  $p_j\to\infty$ as $j\to\infty,$
	\begin{equation}
		\label{eq:step2.1}
		\lim_{j\to\infty} \lam_1(s,p_j)^\frac{1}{p_j}
		=\liminf_{p\to\infty} \lam_1(s,p)^\frac{1}{p},
	\end{equation}
	and for any $j\in\mathbb{N}$ $u_j\in W^{s,p_j}(\Omega),$
	\begin{equation}
		\label{eq:step2.2}
		\|u_{j}\|_{L^{p_j}(\Omega)}=1,\quad 
		\int_\Omega |u_{j}(x)|^{p_j-2} u_{j}(x)\, dx =0,
	\end{equation}
	and 
	\begin{equation}
		\label{eq:step2.2.1}
		\lam_1(s,p_j) =\int_\Omega\int_\Omega  
		\dfrac{|u_j(y)-u_j(x)|^{p_j}}{|x-y|^{n+sp_j}}
		\,dx\,dy.
	\end{equation}
	Then, there exists a constant $C$ independent of $j$ such that
	\begin{equation}
		\label{eq:step2.3}
		[u_j]_{W^{s,p_j}(\Omega)}\le C
	\end{equation}
	for all $j\in\mathbb{N}.$
	
	Let us fix $q\in(1,\infty)$ such that $sq>2n.$ There exists $j_0\in \N$
	such that $p_j\ge q$ for all $j\ge j_0.$ Then
	by H\"older's Inequality, we have that
	\begin{equation}
		\label{eq:step2.4}
		\|u_j\|_{L^q(\Omega)}\le |\Omega|^{\frac1{q}-\frac1{p_j}}
		\|u_j\|_{L^{p_j}(\Omega)}\le |\Omega|^{\frac1{q}-\frac1{p_j}}
		\quad\forall j\ge j_0, 
	\end{equation}
	and taking $r=s-\nicefrac{n}{q}\in(0,1),$ again by H\"older's Inequality, 
	we get
	\begin{equation}
		\label{eq:step2.5}
		\begin{aligned}
			\int_\Omega\int_\Omega
			\dfrac{|u_j(x)-u_j(y)|^q}{|x-y|^{n+rq}}\, dxdy
			&= \int_\Omega\int_\Omega\dfrac{|u_j(x)-u_j(y)|^q}{|x-y|^{sq}}\, 
			dxdy\\
			&\le |\Omega|^{2(1-\frac{q}{p_j})}
			\left(\int_\Omega
			\int_\Omega\dfrac{|u_j(x)-u_j(y)|^{p_j}}{|x-y|^{sp_j}}\,
			 dxdy\right)^{\frac{q}{p_j}}\\
			&\le \mathrm{diam}(\Omega)^{\frac{nq}{p_j}} 
			|\Omega|^{2(1-\frac{q}{p_j})}[u_j]_{W^{s,p_j}(\Omega)}^q.
			\end{aligned}	
	\end{equation}
	Then, by \eqref{eq:step2.3},
	\[
		\int_\Omega\int_\Omega\dfrac{|u_j(x)-u_j(y)|^q}{|x-y|^{n+rq}}\, dxdy
		\le \mathrm{diam}(\Omega)^{\frac{nq}{p_j}} 
		|\Omega|^{2(1-\frac{q}{p_j})}C^q \quad\forall j\ge j_0, 
	\]
	where $C$ is a constant independent of $j.$
	Hence $\{u_j\}_{j\ge j_0}$ is a bounded sequence in $W^{r,q}(\Omega).$
	Then, since $rq=sq-n>n,$ by Theorem \ref{teo:inclucomp}, 
	there exist a subsequence of $\{u_{j}\}_{j\ge j_0},$ which 
	we still denoted 
	by $\{u_{j}\}_{j\ge j_0},$ and a function 
	$u_{\infty}\in C(\overline{\Omega})$ such that
	\begin{align*}
		u_{j}\to u_{\infty} &\quad\mbox{uniformly in } \overline{\Omega},\\
		u_{j}\rightharpoonup u_{\infty} &\quad\mbox{weakly in }
		 W^{r,q}(\Omega).
	\end{align*}
	Then, by \eqref{eq:step2.4}, $\|u_\infty\|_{L^q(\Omega)}\le 
	|\Omega|^{\frac1q},$ and
	by \eqref{eq:step2.1}, \eqref{eq:step2.2.1} and \eqref{eq:step2.5}, we get
	\begin{align*}
		[u_\infty]_{W^{r,q}(\Omega)}&\le
		\liminf_{j\to \infty}[u_j]_{W^{r,q}(\Omega)}\\
		&\le\liminf_{j\to \infty}\mathrm{diam}(\Omega)^{\frac{n}{p_j}} 
			|\Omega|^{2(\frac1{q}-\frac{1}{p_j})}[u_j]_{W^{s,p_j}(\Omega)}\\
		&\le|\Omega|^{\frac2{q}}\liminf_{p\to \infty}
		\lambda_1(s,p)^{\frac{1}{p}}.
	\end{align*}    
	    
	Letting $q\to\infty,$ we get $\|u_\infty\|_{L^{\infty}(\Omega)}\le 1$
	and
	\begin{equation}
	\label{eq:step2.6}
		[u_\infty]_{W^{s,\infty}(\Omega)}\le\liminf_{p\to \infty}
		\lambda_1(s,p)^{\frac{1}{p}}.
	\end{equation}
	
	On the other hand,
	\[
		1=\|u_j\|_{L^{p_ j}(\Omega)}\le|\Omega|^{\frac1{p_j}}
		\|u_j\|_{L^{\infty}(\Omega)}\quad\forall j\ge j_0
	\]
	then $1\le\|u_\infty\|_{L^{\infty}(\Omega)}.$ Hence 
	$\|u_\infty\|_{L^{\infty}(\Omega)}=1$ and by \eqref{eq:step2.6} 
	we get
	\begin{equation}
	\label{eq:step2.7}
		\dfrac{[u_\infty]_{W^{s,\infty}(\Omega)}}
		{\|u_\infty\|_{L^{\infty}(\Omega)}}\le\liminf_{p\to \infty}
		\lambda_1(s,p)^{\frac{1}{p}}.
	\end{equation}
	
	Finally, in \cite{inftyN} it was proved that the condition 
	$\int_\Omega |u_{j}(x)|^{p_j-2} u_{j}(x)\, dx =0$ leads to
	$\sup u_{\infty} +\inf u_{\infty}=0.$ Then, using \eqref{eq:step2.7},
	we get
	\[
		\inf\left\{\dfrac{[u]_{W^{s,\infty}(\Omega)}}
		{\|u\|_{L^{\infty}(\Omega)}}\colon
		u\in \mathcal{A} \right\}\le \liminf_{p\to \infty}
		\lambda_1(s,p)^{\frac{1}{p}}.
	\]

	{\it Step 3.} Finally, we prove that 
	\begin{equation}
	\label{eq:step3}
		\dfrac2{\mathrm{diam}(\Omega)^s}
		\le \inf\left\{\dfrac{[u]_{W^{s,\infty}(\Omega)}}
		{\|u\|_{L^{\infty}(\Omega)}}\colon
		u\in \mathcal{A} \right\}.
	\end{equation}
	
	For any $u\in\mathcal{A},$ we have 
	\begin{align*}
		2\|u\|_{L^\infty(\Omega)} &= 
		\sup u - \inf u\\
		 &= \sup\{ |u(x)-u(y)|\colon x,y \in \Omega\}\\
		 &=\sup\left\{ 
		 |x-y|^s\dfrac{|u(x)-u(y)|}{|x-y|^s}
		 \colon x,y \in \Omega\right\}\\
		&\leq
		\mathrm{diam}(\Omega)^s [u]_{W^{s,\infty}(\Omega)}.
	\end{align*}
	Thus
	\[
		\dfrac2{\mathrm{diam}(\Omega)^s}
		\le \dfrac{[u]_{W^{s,\infty}(\Omega)}}
		{\|u\|_{L^{\infty}(\Omega)}}
	\]
	for all $u\in\mathcal{A}.$ Hence \eqref{eq:step3} holds.

	\medskip
	
	Then, by steps 1--3, we get
	\begin{align*}
	 	\dfrac2{\mathrm{diam}(\Omega)^ s}
		&\le \inf\left\{\dfrac{[u]_{W^{s,\infty}(\Omega)}}
		{\|u\|_{L^{\infty}(\Omega)}}\colon
		u\in \mathcal{A} \right\}\\
		&\le \liminf_{p\to \infty}
		\lambda_1(s,p)^{\frac{1}{p}}\\
		&\le \limsup_{p\to \infty}
		\lambda_1(s,p)^{\frac{1}{p}}\\
		&\le\dfrac2{\mathrm{diam}(\Omega)^s},
	\end{align*}
	that is
	\[
		\lim_{p\to\infty} \lam_1(s,p)^{\frac{1}{p}
		}=\frac{2}{\mathrm{diam}(\Omega)^s}=
		\inf\left\{\dfrac{[u]_{W^{s,\infty}(\Omega)}}
		{\|u\|_{L^{\infty}(\Omega)}}\colon
		u\in \mathcal{A} \right\}.
	\]
	In addition, by \eqref{eq:step2.7}, we have that 
	$u_\infty$ is a minimizer of $\lambda_1(1,\infty)$
	which proves the lemma.
\end{proof}
 
Our last aim is to show that $u_\infty$ is a viscosity solution of
\eqref{ecuvisc}. We start by intruding the definition of viscosity solution.

\begin{defn} Suppose that $u\in C(\Omega)$. We say that $u$ is a 
\emph{viscosity super-solution} (resp. \emph{viscosity sub-solution})  
in $\Omega$ of the equation \eqref{ecuvisc} if the following holds: 
whenever $x_0\in\Omega$ and $\varphi\in C^1(\overline{\Omega})$ are such that
$$
	\varphi(x_0)=u(x_0) \quad \mbox{ and } \quad \varphi(x)\leq u(x) \quad (resp.\, \varphi(x)\geq u(x))\quad \mbox{ for all }\quad x\in \R^n
$$
then we have
\[ 
			\begin{cases}
				\max\{\LL_{s,\infty} \varphi 
				(x_0), \LL_{s,\infty}^- \varphi(x_0) 
				+\lam_1(1,\infty) \varphi(x_0)\}\le 0 
				\, (\mbox{resp.} \ge 0) 
				\qquad & \mbox{if }\varphi(x_0)>0\\
				\LL_{s,\infty} \varphi(x_0)\le0  \, (\mbox{resp.} \ge 0)
				\qquad & \mbox{if }\varphi(x_0)=0\\
				\min\{\LL_{s,\infty} \varphi(x_0), \LL_{s,\infty}^+ 
				\varphi(x) +
				\lam_1(1,\infty) \varphi(x_0)\}
				\le0  \, (\mbox{resp.} \ge 0) \qquad & \mbox{if }\varphi(x_0)<0.\\
			\end{cases}
\]
A \emph{viscosity solution}  is defined as being both a viscosity super-solution and a viscosity sub-solution.
\end{defn}

For the proof of the following lemma we borrow ideas from \cite[Theorem 23]{LL}.

\begin{lema}\label{lema:euler-lagrange}
	Let $\Omega$ be  bounded open connected domain in $\R^n$ and 
	$s\in(0,1)$. 
	Then $u_\infty$ is a solution of \eqref{ecuvisc} in the viscosity sense.
\end{lema}

\begin{proof}
	We begin by observing that, by Lemma \ref{lema:ptoinfty1}, 
	$u_\infty$ is a minimizer of
	$\lambda_1(1,\infty)$ and there exists
	a  sequence $\{p_j\}_{j\in\N}$ such that $p_j\to\infty$  and 
	$u_j\to u_\infty$ uniformly in $\overline{\Omega}$ as $j\to\infty,$ 
	where 
	$u_j$ is an eigenfunction associated to $\lambda_1(s,p_j).$ 
	Without loss of generality, we can assume that $p_js>n$ for all 
	$j\in\mathbb{N}.$ Then
	$u_j\in C(\overline{\Omega})$ for all $j\in\mathbb{N}.$
	
	\medskip
	
	We only verify that $u_\infty$ is a viscosity 
	super-solution of \eqref{ecuvisc}. The proof that 
	$u_\infty$ is also a sub-solution is similar. Let us fix
	some point $x_0\in\Omega.$ We assume that $\varphi$ is a test function
	touching $u_\infty$ from below at a point $x_0$, and we may assume that the 
	touching is strict by considering $\varphi(x)-|x|^2\eta(x)$, 
	where $\eta=1$ in a neighborhood of $x_0$ and $\eta\geq 0$. It follows
	that $u_j-\varphi$ attains its minimum at points $x_j\to x_0$. 
	By adding a suitable constant $c_j$ we can arrange it so that 
	$\varphi+c_j$ touches $u_j$ from below at the point $x_j.$
	
	\medskip

    By Theorem \ref{debilviscosa}, a eigenfunction is a viscosity solution
    of \eqref{eq:visco}, then we have
	$$
		\LL_{s,p_j} \varphi(x_j)+\lam_1(s,p_j) u_j^{p_j-1}(x_j) \leq 0.
	$$	
	We write the last inequality as
	$$
		A_j^{p_j-1} - B_j^{p_j-1} +C_j^{p_j-1}-D_j^{p_j-1} \leq 0
	$$
	where
	\begin{align*}
		A_j^{p_j-1}&=2\int_\Omega 
		\frac{|\varphi(y)-\varphi(x_j)|^{p_j-2}
		(\varphi(y)-\varphi(x_j))^+}{|y-x_j|^{n+sp_j}}\, dy,\\
		B_j^{p_j-1}&=2\int_\Omega 
		\frac{|\varphi(y)-\varphi(x_j)|^{p_j-2}(\varphi(y)-
		\varphi(x_j))^-}{|y-x_j|^{n+sp_j}}\, dy,\\	
		C_j^{p_j-1}&=\lam_1(s,p_j) (u_{j}^+(x_j))^{p_j-1},\\
		D_j^{p_j-1}&=\lam_1(s,p_j) (u_{j}^-(x_j))^{p_j-1}.
	\end{align*}

	In \cite[Lemma 6.5]{CLM}, it is proved that
	\[
		A_j\to \LL^+_{s,\infty} \varphi(x_0), 
		\quad \qquad B_j\to -\LL^-_{s,\infty} \varphi(x_0),
	\]
	as $j\to \infty.$ In addition, by Lemma \ref{lema:ptoinfty1}, we have
	\[
		C_j\to  \lam_1(s,\infty)\varphi(x_0)^+, 
		\quad D_j\to \lam_1(s,\infty)\varphi(x_0)^-.	
	\]

	On the other hand, if $u_\infty(x_0)>0$ we get
	$$
		A_j^{p_j-1} + C_j^{p_j-1} \leq B_j^{p_j-1},
	$$
	and by dropping either $A_j^{p_j-1}$ or $C_j^{p_j-1}$, 
	and sending $j\to \infty$  we see that
	$$
		\LL_{s,\infty}^+ \varphi(x_0) \leq - 
		\LL_{s,\infty}^- \varphi(x_0) \quad \mbox{ and } \quad 
		\lam_1(s,\infty)\varphi(x_0)^+ \leq -\LL_{s,\infty}^- \varphi(x_0),
	$$ 
	which leads to
	$$
		\LL_{s,\infty} \varphi(x_0) \leq 0 \quad \mbox{ and } \quad 
		\LL_{s,\infty}^- \varphi(x_0) + \lam_1(s,\infty)\varphi(x_0)^+\leq 0,
	$$ 
	and we can write
	\[
		\max\{\LL_{s,\infty} \varphi(x_0), \LL_{s,\infty}^- 
		\varphi(x_0) + \lam_1(s,\infty)\varphi(x_0)^+\} \leq 0.
	\]

	If $u_\infty(x_0)<0$ we obtain that
	$$
		A_j^{p_j-1} \leq  D_j^{p_j-1} + B_j^{p_j-1} 
		\leq 2 \max\{ B_j^{p_j-1}, D_j^{p_j-1}\},
	$$
	that is
	$$
		A_j\leq 2^{\frac1{p_j-1}} \max\{ B_j, D_j\}.
	$$
	Then, sending $j\to\infty$, we get
	$$
		\LL_{s,\infty} \varphi(x_0) \leq 0 \quad \mbox{ or } \quad 
		\LL_{s,\infty}^+ \varphi(x_0) - \lam_1(s,\infty)\varphi(x_0)^-\leq 0,
	$$ 
	which can be written as
	\[
		\min\{\LL_{s,\infty} \varphi(x_0), 
		\LL_{s,\infty}^+ \varphi(x_0) - \lam_1(s,\infty)\varphi(x_0)^-\} 
		\leq 0.
	\]

	Finally if $u_\infty(x_0)=0$, 
	it follows that $\LL_{s,\infty} \varphi(x_0)\leq0$.
	This proves that $u_\infty$ is a viscosity super-solution of equation 
	\eqref{ecuvisc}.
\end{proof}
%
%
%
%

\section{Comments}\label{comments}

Let $d(\cdot,\cdot)$ be a distance equivalent to the usual distance.
If we take the following non-linear non-local operator 
\[
	\mathfrak{L}_{s,p}u(x)\coloneqq 2\mbox{ p.v.}\int_\Omega 
	\frac{|u(y)-u(x)|^{p-2}(u(y)-u(x))}{d(x,y)^{n+sp}} 
	\, dy,
\]
in place of $\LL_{s,p},$ following what was done in the previous section,
we can see that the first non-zero eigenvalue of 
\[
	\begin{cases}
	-	\mathfrak{L}_{s,p}u = \lam |u|^{p-2}u \quad \mbox{ in } \Omega, \\
		u\in W^{s,p}(\Omega),
	\end{cases}
\]
is 
\begin{equation*}
            \lambda_1^d(s,p)\coloneqq
            \inf\left\{
            \dfrac{\displaystyle
            \int_\Omega \int_\Omega \frac{|u(x)-u(y)|^p}
            {d(x,y)^{n+sp}} \, 
	  dx \,dy}{ \displaystyle\int_\Omega|u(x)|^p\, dx}\colon u\in 
	  \mathcal{X}_{s,p}
            \right\}.
\end{equation*}
 
Moreover
\[
	\lim_{p\to\infty} \left(\lambda_1^d(s,p\right))^{\frac{1}{p}}
		=\frac{2}{\mathrm{diam}_d(\Omega)^s}=
		\lambda_1^d(s,\infty)\coloneqq 
		\inf\left\{\dfrac{[u]_{d,W^{s,\infty}(\Omega)}}
		{\|u\|_{L^{\infty}(\Omega)}}\colon
		u\in \mathcal{A} \right\}.
\]
where 
\[
	[u]_{d,W^{s,\infty}(\Omega)}=\sup\left\{
	\dfrac{|u(x)-u(y)|}{d_\Omega(x,y)^s}\colon x, y\in\Omega\right\}
\]
and $\mathrm{diam}_d(\Omega)=\sup\{d(x,y)\colon x,y\in\Omega\}.$
 
\medskip
 
Finally, observe that if $d$ is the geodesic distance inside $\Omega$
then $\mathrm{diam}_d(\Omega)$ is the intrinsic diameter as in the local
case.

\section*{Acknowledgement}
We want to thank to Prof. Nicolas Saintier and Prof. Julio Rossi for their 
comments that  helped us to improve Section 5.
 
\bibliographystyle{amsplain}
\bibliography{biblio}

\end{document}